\documentclass[10pt,reqno]{amsart}
\usepackage{ucs}
\usepackage{amsmath, amssymb, amscd}
\usepackage{pb-diagram}
\usepackage[latin1]{inputenc}
\usepackage{amsmath}
\usepackage{amssymb}
\usepackage{epsfig}
\usepackage{graphicx}
\usepackage{psfrag}
\usepackage{xypic}

\newtheorem{theorem}{Theorem}[section]
\newtheorem{definition}[theorem]{Definition}
\newtheorem{lemma}[theorem]{Lemma}

\newtheorem{proposition}[theorem]{Proposition}
\newtheorem{corollary}[theorem]{Corollary}

\newtheorem{conjecture}[theorem]{Conjecture}

\renewcommand{\epsilon}{\varepsilon}

\DeclareMathAlphabet{\mathpzc}{OT1}{pzc}{m}{it}
\usepackage{mathrsfs}

\newcommand{\ol}{\overline}

\newcommand{\Z}{\mathbb{Z}}
\newcommand{\C}{\mathbb{C}}

\renewcommand{\qed}{$\hfill \square$ \smallskip \\}
\renewcommand{\phi}{\varphi}

\newcommand{\F}{\mathscr{F}}

\newcommand{\hash}{\#}

\begin{document}
\thispagestyle{empty}
\title[A note on Gornik's perturbation of Khovanov-Rozansky homology]{A note on Gornik's perturbation of Khovanov-Rozansky homology}
\author{Andrew Lobb}
\email{lobb@math.sunysb.edu}
\address{Mathematics Department \\ Stony Brook University \\ Stony Brook NY 11794 \\ USA}

\begin{abstract}
We show that the information contained in the associated graded vector space to Gornik's version of Khovanov-Rozansky knot homology is equivalent to a single even integer $s_n(K)$.  Furthermore we show that $s_n$ is a homomorphism from the smooth knot concordance group to the integers.  This is in analogy with Rasmussen's invariant coming from a perturbation of Khovanov homology.
\end{abstract}

\maketitle

\section{Introduction and statement of results}

In the last few years there have been associated to a knot $K \subset S^3$ various multiply-graded modules, each one exhibiting a classical knot polynomial as its graded Euler characteristic.  It now seems likely that such \emph{knot homologies} exist for each polynomial arising from the Reshetikhin-Turaev construction.

It has been observed that there sometimes exist spectral sequences converging from one knot homology to another (see \cite{Ras2} for a slew of these).  One of the first examples was due to Lee \cite{Lee}.

\subsection{Khovanov homology and Lee's spectral sequence}

From here on we shall work over the complex numbers $\C$.  The $E_2$ page of Lee's spectral sequence is standard Khovanov homology \cite{K1}.  With a one-component knot $K$ as input, the $E_\infty$ page is a $2$-dimensional complex vector space supported in homological degree $0$.  The $E_\infty$ page also has another integer grading (the quantum grading), let us write $\widetilde{H}^{i,j}(K)$ for this $E_\infty$ page, where $i$ is the homological grading and $j$ is the quantum grading.  Another way to think of $\widetilde{H}^{i,j}(K)$ is as the associated graded vector space to the homology of a filtered chain complex defined by Lee.

Rasmussen \cite{Ras1} showed that $\widetilde{H}^{i,j}(K)$ is supported in bidegrees $i=0, j= s-1$ and $i=0, j= s+1$ where $s(K) \in 2\Z$.  Hence the information contained in $\widetilde{H}^{i,j}(K)$ is equivalent to a single even integer.  Rasmussen further showed that

\begin{theorem}[Rasmussen \cite{Ras1}]
\label{rasthm1}
Let $g_*(K)$ be the smooth slice genus of the knot $K$, then
 
\[ g_*(K) \geq \frac{|s(K)|}{2} \rm{.} \]
\end{theorem}

\noindent This bound is sufficient to recover the Milnor conjecture on the slice genus of torus knots, a result previously only accessible through gauge theory.  Furthermore Rasmussen showed

\begin{theorem}[Rasmussen \cite{Ras1}]
\label{rasthm2}
The map $s: K \mapsto s(K) \in 2\Z$ is a homomorphism from the smooth concordance group of knots to the integers.
\end{theorem}

\subsection{Khovanov-Rozansky homology and Gornik's spectral sequence}

In the case of Khovanov-Rozansky homology $H^{i,j}_n(K)$ ($n \geq 2$) (which has the quantum $sl(n)$ knot polynomial as its Euler characteristic), a spectral sequence with $E_2$ page $H^{i,j}_n(K)$ was defined by Gornik \cite{G}.  He showed that the $E_\infty$ page of this spectral sequence is a complex vector space of dimension $n$ supported in homological degree~$i=0$.  The invariance of this spectral sequence under the Reidemeister moves was first shown by Wu \cite{Wu1}.

Again there is also a quantum grading on this vector space, and the vector space can be thought of as the associated graded vector space to the homology of a filtered chain complex $\F^j \widetilde{C}_n^i (D)$ defined by Gornik for any diagram $D$ of a knot $K$.  We shall write $\F^j\widetilde{H}_n^i(K)$ for the filtered homology groups $\ldots \subseteq \F^{j-1}\widetilde{H}_n^i(K) \subseteq \F^{j}\widetilde{H}_n^i(K) \subseteq \ldots$ of this chain complex and

\[\widetilde{H}^{i,j}_n(K) = \F^{j}\widetilde{H}_n^i(K) / \F^{j-1}\widetilde{H}_n^i(K) \]

\noindent for the associated graded vector space.

It was shown by the author \cite{L1} and independently by Wu \cite{Wu1} that one can extract a lower-bound on the slice genus from the quantum $j$-grading of each non-zero vector space $\widetilde{H}^{0,j}_n(K)$ (in fact in these cited papers this was done also for more general perturbations of Khovanov-Rozansky homology than Gornik's).  Again, these lower-bounds are enough to imply the Milnor conjecture on the slice genus of torus knots.  The highest non-zero quantum grading in this set-up has been called $g_n^{{\rm max}}$ and the lowest $g_n^{{\rm min}}$ by Wu.  In \cite{Wu2} Wu asks for a relation between $g_n^{{\rm max}}$ and $g_n^{{\rm min}}$, we provide an answer with our Theorem \ref{mainthm1}.

\subsection{New results}
In the current paper we first show that the information contained in $\widetilde{H}^{i,j}_n(K)$ is equivalent to a single even integer $s_n(K)$.

\begin{theorem}
\label{mainthm1}
For $K$ a knot define the polynomial

\[ \widetilde{P}_n(q) = \sum_{j= -\infty}^{j= \infty} \dim_{\C}(\widetilde{H}^{0,j}_n(K)) q^j {\rm .} \]

\noindent Then there exists $s_n(K) \in 2 \Z$ such that

\[ \widetilde{P}_n(q) = q^{s_n(K)} \frac{(q^n - q^{-n})}{(q - q^{-1})} {\rm .} \]
\end{theorem}

\noindent In other words, this theorem says that the Gornik homology of any knot $K$ is isomorphic to that of the unknot, but shifted by quantum degree $s_n(K)$.

The results of the author and of Wu's on the slice genus are then immediately stated as the following:

\begin{corollary}[Lobb \cite{L1}, Wu \cite{Wu1}]
\label{lobbwu}
Writing $g_*(K)$ for the smooth slice genus of a knot, we have

\[ g_*(K) \geq \frac{|s_n(K)|}{2(n-1)} {\rm .}\]

Furthermore, if $K$ admits a diagram $D$ with only positive crossings then

\begin{eqnarray*}
g_*(K) &=& \frac{-s_n(K)}{2(n-1)}  \\
&=& \frac{1}{2}(1 - \hash O(D) + w(D)) {\rm ,}
\end{eqnarray*}

\noindent where $\hash O(D)$ is the number of circles in the oriented resolution of $D$ and $w(D)$ is the writhe of $D$.
\end{corollary}

It is a question of much interest whether the $s_n(K)$ are in fact all equivalent to each other.  We hope that this is not true, and do not know whether to expect it to be true.  Nevertheless, let us formulate this as a conjecture.

\begin{conjecture}
For any knot $K$ and $m,n \geq 2$ we have

\[ \frac{s_m(K)}{s_n(K)} = \frac{m-1}{n-1} {\rm .} \]

\noindent We note that $s_2 (K) = -s(K)$ so that every $s_n$ is equivalent to Rasmussen's original~$s(K)$.
\end{conjecture}

The falsity of this conjecture would have consequences for the non-degeneracy of the spectral sequences defined by Rasmussen \cite{Ras2} on the triply-graded Khovanov-Rozansky homology \cite{KR2}.  We are involved in a program with Daniel Krasner to try to find a counterexample to this conjecture.  One can also make a weaker conjecture:

\begin{conjecture}
For any knot $K$ and $n \geq 2$ we have

\[ s_n(K) \in 2(n-1) \Z {\rm .} \]
\end{conjecture}

\noindent This has the appeal that it would rule out the possibility of fractional bounds on the slice genus coming from Corollary \ref{lobbwu}, but again we have no expectations either way on the truth of this conjecture.

By analogy with Rasmussen's Theorem \ref{rasthm2} we might anticipate that each $s_n$ is a concordance homomorphism.  We show that this is in fact the case:

\begin{theorem}
\label{mainthm2}
For each $n \geq 2$, the map $s_n: K \mapsto s_n(K) \in 2\Z$ is a homomorphism from the smooth concordance group of knots to the integers.
\end{theorem}

\noindent This theorem tells us that we have a concordance homomorphism for each integer $\geq 2$.  It is a fascinating problem to try and understand if and how these homomorphisms are related to each other; we hope that this paper will stimulate some activity towards this goal.

We conclude by noting that there are many properties of Rasmussen's concordance homomorphism $s$ from Khovanov homology and of the homomorphism $\tau$ coming from Heegaard-Floer knot homology \cite{Ras3} \cite{OzSz} which follow formally from the properties of $s$ and $\tau$ analogous to Corollary \ref{lobbwu} and Theorem \ref{mainthm2}.  Rescaled versions of these results can now be seen to hold for $s_n$.  We restrict ourselves to mentioning one of these which is not well-known as following from these formal properties.

\begin{corollary}
\label{lobbtomomi}
If $K$ is an alternating knot then

\[ s_n(K) = \frac{1}{1-n} \sigma(K) {\rm ,} \]

\noindent where $\sigma(K)$ is the classical knot signature of $K$.
\end{corollary}

\noindent We sketch the proof of this at the end of the next section.

\section{Proofs of results}

We assume in this section some familiarity with \cite{KR1} by Khovanov and Rozansky.  We fix an integer $n \geq 2$ and let $K$ be a $1$-component knot.  In \cite{KR1} the polynomial $w = x^{n+1}$ is called the \emph{potential}.  Gornik's key insight \cite{G} was that it made sense to take a perturbation $\widetilde{w} = x^{n+1} - (n+1)x$ of this potential and much of \cite{KR1} goes through as before.  Gornik showed that for his choice of potential $\widetilde{w}$, a knot diagram $D$ determines a chain complex that no longer has a quantum grading but instead a quantum filtration respected by the differential.

\[ \ldots \subseteq \F^{j-1}\widetilde{C}_n^i(D) \subseteq \F^{j}\widetilde{C}_n^i(D) \subseteq \ldots {\rm ,}\]

\[ d : \F^{j}\widetilde{C}_n^i(D) \rightarrow \F^{j}\widetilde{C}_n^{i+1}(D)  \rm{.}\]

\noindent It was immediate from his definitions that there exists a spectral sequence with $E_2$ page the original Khovanov-Rozansky homology $H_n^{i,j}(K)$ converging to the associated graded vector space

\[ E_\infty^{i,j}(K) = \widetilde{H}_n^{i,j}(K) = \F^{j}\widetilde{H}_n^i(K) / \F^{j-1}\widetilde{H}_n^i(K)\]

\noindent to the filtered homology groups $\F^{j}\widetilde{H}_n^i(K)$.

Given a knot diagram $D$ for $K$, Gornik gave a basis at the chain level generating the homology; we now describe this basis.  We write $O(D)$ for the oriented resolution of $D$, and write $r$ for the number of components of $O(D)$.  The oriented resolution $O(D)$ gives rise to a summand of the chain group $\widetilde{C}_n^0(D) = \cup_j \F^j \widetilde{C}_n^0(D)$, isomorphic in a natural way to

\[ \C[x_1,x_2, \ldots, x_r]/(x_1^n - 1, x_2^n - 1, \ldots, x_r^n - 1) \, [(1-n) ({w}(D) + r) ] {\rm ,} \]

\noindent where we have indicated a shift in the quantum filtration depending on $r$ and on the writhe $w(D)$ of the diagram.

\begin{definition}
\label{kvetch}
Let $\xi = e^{2 \pi i / n}$.  For each $p = 0,1, \ldots , n-1$ we define an element $g_p \in \widetilde{C}_n^0(D)$ that lies in this summand by

\[g_p = \prod_{k=1}^r \frac{(x_k^n - 1)}{(x_k - \xi^p)} {\rm .}\]
\end{definition}

Then we know that:

\begin{theorem}[Gornik \cite{G}]
Each $g_p$ is a cycle and $\{[g_0], [g_1], \ldots, [g_{n-1}]\}$ is a basis for the homology $\widetilde{H}^i_n(K) = \cup_j \F^{j}\widetilde{H}_n^i(K)$.  Consequently $\widetilde{H}^i_n(K)$ is a vector space of dimension $n$ supported in homological degree $i=0$.
\end{theorem}

Our first observation is that we can find a good basis for the subspace of $\widetilde{C}_n^0(D)$ spanned by $g_0, g_1, \ldots, g_{n-1}$.  What we mean here by `good' requires another definition.

\begin{definition}
A monomial $\prod_{i=1}^{s} x_i^{a_i} \in \C[x_1,x_2, \ldots, x_s]$ is said to be of $n$-degree $d$ iff

\[ \sum_{i=1}^s a_i = d \pmod{n} {\rm .}\]

\noindent A polynomial is said to have be $n$-homogenous of $n$-degree $d$ iff it is a linear combination of monomials of $n$-degree $d$.
\end{definition}

We note that projection extends the notion of $n$-degree unambiguously to elements lying in the ring

\[\C[x_1,x_2, \ldots, x_s]/(x_1^n - 1, x_2^n - 1, \ldots, x_s^n - 1) \]

\noindent since the quotient ideal is generated by $n$-homogeneous polynomials.

Next we give a basis consisting of $n$-homogeneous elements for the vector space spanned by the elements $g_0, g_1, \ldots, g_{n-1} \in \widetilde{C}_n^0(D)$.

\begin{lemma}
\label{ronen}
Let $g_0, g_1, \ldots, g_{n-1}$ be given as in Definition \ref{kvetch}, and consider the $n$-dimensional complex vector space

\[ V = < g_0, g_1, \ldots, g_{n-1} > \subseteq  \C[x_1,x_2, \ldots, x_r]/(x_1^n - 1, x_2^n - 1, \ldots, x_r^n - 1) {\rm .} \]

\noindent For $p=0, 1, \ldots, n-1$ let 

\[h_p \in \C[x_1,x_2, \ldots, x_r]/(x_1^n - 1, x_2^n - 1, \ldots, x_r^n - 1) \]

\noindent be the unique $n$-homogeneous element of $n$-degree $p$ such that

\[ g_0 = h_0 + h_1 + \cdots + h_{n-1} {\rm .} \]

\noindent Then we have

\[ V = < h_0, h_1, \ldots, h_{n-1} > {\rm .} \]
\end{lemma}

\begin{proof}
For dimensional reasons it is enough to show that for each $t = 0,1, \ldots, n-1$ we have

\[ g_t \in < h_0, h_1, \ldots, h_{n-1} > {\rm .} \]

\noindent So let us fix such a $t$ and let $\ol{h}_p$ be the unique $n$-homogeneous element of $n$-degree $p$ such that

\[ g_t = \ol{h}_0 + \ol{h}_1 + \cdots + \ol{h}_{n-1} \rm{.} \]

\noindent We will show that $\ol{h}_p$ is a multiple of $h_p$ and then we will be done.

Consider a monomial of $n$-degree $p$

\[ \prod_{i=1}^{r} x_i^{a_i} \, \, {\rm where} \, \,  \sum_{i=1}^r a_i = p \pmod{n} \, \, {\rm and} \, \, 0 \leq a_i \leq n-1 \, \, \forall i{\rm .} \]

\noindent The coefficient of this monomial in $h_p$ (or, equivalently, in $g_0$) is clearly $1$.  The coefficient $c$ of this monomial in $g_t$ is expressible as a product $c = c_1c_2 \cdots c_r$ where $c_i$ is the coefficient of $x^{a_i}$ in the expansion of 

\[ \frac{x^n - 1}{x - \xi^t} = \frac{x^n - (\xi^t)^n}{x - \xi^t} {\rm .} \]

\noindent We leave it to the reader to check that $c_i = \xi^{-(a_i + 1)t}$, so that

\[ c = \xi^{ -t (\sum_{i=1}^r (a_i + 1 ))} = \xi^{-t (p + r)} {\rm .} \]

\noindent Hence we see that

\[ \ol{h}_p = \xi^{-t (p+r)} h_p \, \, \, {\rm so} \, \, \,  g_t \in < h_0, h_1, \ldots, h_{n-1} > {\rm .}\]
\end{proof}

To put our new $n$-homogeneous basis to use, we require a proposition telling us how we might expect it to behave with respect to the filtration.  In what follows, since we are assuming some familiarity with \cite{KR1}, we allow ourselves to refer to a matrix factorization as just a letter, $M$.  We begin with a definition.

\begin{definition}
If $V$ is some filtered vector space

\[ \cdots \subseteq \F^i V \subseteq \F^{i+1}V \subseteq \cdots {\rm ,} \]

\noindent and we have a non-zero $x \in V$, we shall define the \emph{quantum grading} ${\rm qgr}(x) \in \Z$ by the requirement that $x$ is non-zero in

\[ \F^{{\rm qgr}(x)}V / \F^{{\rm qgr}(x) -1}V {\rm .} \]

\end{definition}

\noindent The reason for the word `quantum' in the definition is that in this paper the only vector spaces we shall worry about are those coming from chain groups or homology groups carrying a `quantum' filtration.

\begin{proposition}
\label{sabin}
If $M$ is a matrix factorization whose homology $H(M)$ appears as a summand of the chain group $\widetilde{C}^i(D)$, then there is a natural $(\Z / 2n\Z)$-grading on $H(M)$ which we write as

\[ {\rm Gr}^\alpha H(M) \, \, {\rm for} \, \, \alpha \in \Z / 2n\Z {\rm .} \]

\noindent This grading extends to a grading on the chain groups $\widetilde{C}_n^i(D)$, which is respected by the differential

\[ d : {\rm Gr}^\alpha \widetilde{C}_n^i(D) \longrightarrow  {\rm Gr}^\alpha \widetilde{C}_n^{i+1}(D) \, \, {\rm for} \, \, \alpha \in \Z / 2n\Z {\rm ,} \]

\noindent thus giving a $(\Z / 2n \Z)$-grading on the homology groups ${\rm Gr}^\alpha \widetilde{H}_n^i (K)$ for $\alpha \in \Z / 2n\Z$.

Furthermore, if $a \in {\rm Gr}^\alpha \widetilde{C}_n^0(D)$ and $b \in {\rm Gr}^\beta \widetilde{C}_n^0(D)$ represent non-zero classes $[a]$, $[b]$ in homology $\widetilde{H}_n^0(K)$ then we have

\begin{eqnarray*}
\alpha - \beta &=& {\rm qgr}(a) - {\rm qgr}(b) \pmod{2n} \\
&=& {\rm qgr}([a]) - {\rm qgr}([b])  \pmod{2n} \rm{.}
\end{eqnarray*}
\end{proposition}

\begin{proof}
The matrix factorization $M$ consists of two `internal' graded vector spaces $V_0$, $V_1$ and pair of `internal' differentials

\[ d_0 : V_0 \rightarrow V_1 \, \, {\rm and} \, \, d_1 : V_1 \rightarrow V_0, d_1d_0 = d_0d_1 = 0 {\rm .} \]

\noindent If we were working with Khovanov and Rozansky's potential $w = x^{n+1}$ then we would know that these internal differentials $d_0$, $d_1$ were both graded of degree $n+1$.  But with Gornik's potential $\widetilde{w} = x^{n+1} - (n+1)x$ the internal differentials cease to respect the grading.  So instead we take the filtration associated to the grading of the internal vector spaces and we observe that the internal differentials are then \emph{filtered} of degree $n+1$.  This gives rise to a filtered homology $H(M)$ and so to filtered chain groups.

The crux of this proposition is recognizing that the polynomials appearing as matrix entries in Gornik's internal differentials are all $n$-homogenous.  Since the various $x_i$ appearing in the definition of $M$ are assigned grading $2$, this means that the homology $H(M)$ inherits a $(\Z / 2n \Z)$-grading from the $(\Z / 2n \Z)$-grading on the internal vector spaces of $M$ coming from collapsing their $\Z$-grading.

Similarly the differentials on the chain complex $\widetilde{C}_n^i(D)$ have $n$-homogeneous matrix entries.  It needs to be checked that these entries are graded of degree $0 \in \Z / 2n \Z$ - we leave this to the reader.  Hence we inherit a $(\Z / 2n\Z)$-grading on homology 

\[ {\rm Gr}^\alpha \widetilde{H}_n^i (K) \, \, {\rm where} \, \, \alpha \in \Z / 2n\Z {\rm .} \]

The first equality of the final part of the proposition follows from the observation that both the filtration and the $(\Z / 2n\Z)$-grading on $\widetilde{C}_n^i$ are induced from the same $\Z$-grading on the matrix factorizations.  The second equality follows from the fact that the differential on $\widetilde{C}_n^i$ respects the $(\Z / 2n \Z)$-grading.
\end{proof}

In Proposition \ref{sabin} we restricted ourselves to \emph{relative} quantum gradings, but we did this simply as a matter of convenience, so that we did not have to worry about the various grading shifts happening in the definition of the chain complex.  It is of course possible to more precise.  The content of the next proposition is that we have figured out the grading shifts to give a precise statement of Proposition \ref{sabin} applied to the case of our $n$-homogenous generators $h_0, h_1, \ldots, h_{n-1}$.

\begin{proposition}
\label{sonal}
For $p=0,1,\ldots,n-1$, each $h_p$ of Lemma \ref{ronen} can be considered as a cycle of the chain group $\widetilde{C}_n^0(D)$, each lying in the summand of this chain group corresponding to the oriented resolution $O(D)$.

Then each $[h_p]$ is a non-zero class in homology lying in the graded part $\widetilde{H}_n^{0,j_p}(K)$ for some $j_p$ satisfying

\[ j_p = 2p + (1-n)(w(D) + r) \pmod{2n} {\rm .} \] 
\end{proposition}

\begin{proof}
Certainly each $h_p$ lies in a unique $(\Z / 2n \Z)$-grading.  We note that the writhe of the diagram $w(D)$ and the number $r$ of components of $O(D)$ appear in Proposition \ref{sonal} because of the grading shift of the chain group summand.  The factors of $2$ appear since the various $x_i$ appearing in the definition of the homology are assigned grading $2$.  We note also that $w(D) + r$ is always an odd number.
\end{proof}

\begin{definition}
For $K$ a knot let

\[ s_n^{\rm max}(K) = {\rm max}\{ j : \widetilde{H}_n^{0, j} (K) = \C \} {\rm ,} \]

\noindent and

\[ s_n^{\rm min}(K) = {\rm min}\{ j : \widetilde{H}_n^{0, j} (K) = \C \} {\rm .} \]
\end{definition}

It is now clear that Theorem \ref{mainthm1} follows immediately from Proposition \ref{sonal} and the following:

\begin{proposition}
\label{kwong}
For any knot $K$ we have

\[ s_n^{\rm max}(K) - s_n^{\rm min}(K) \leq 2(n-1) \rm{.} \]
\end{proposition}

\noindent To verify Proposition \ref{kwong} we need to appeal to the results of \cite{L1}, specifically those of Subsection $3.3$ which explains how, given a link $L$, $\widetilde{H}_n^{i,j}(L)$ may change under elementary $1$-handle addition to $L$.  We do not need these results in full generality; the relevant picture for this paper is that of Figure~\ref{connectsum}.

\begin{figure}
\centerline{
{
\psfrag{K1}{$K_1$}
\psfrag{K2}{$K_2$}
\psfrag{1h}{$1{\rm -handle}$}
\psfrag{ldots}{$\ldots$}
\psfrag{K}{$K$}
\psfrag{T-(D)}{$T^-(D)$}
\psfrag{T+(D)}{$T^+(D)$}
\includegraphics[height=2.5in,width=3.7in]{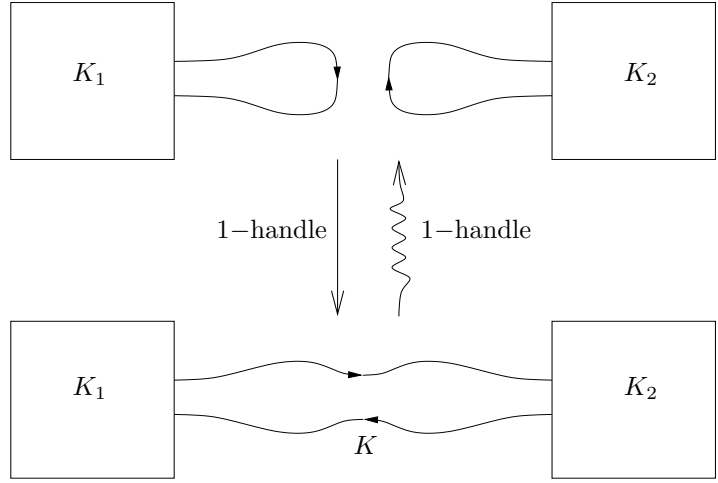}
}}
\caption{In this figure we show how the connect sum $K = K_1 \hash K_2$ of two knots $K_1$, $K_2$ is obtained from the disjoint union of the knots by a knot cobordism consisting of a single $1$-handle attachment (the straight arrow), and likewise the reverse direction (the bendy arrow).}
\label{connectsum}
\end{figure}

We state the next proposition without proof and refer interested readers to Subsection $3.3$ of \cite{L1} for more details.

\begin{proposition}
Consider the set-up of Figure $1$ where $K = K_1 \hash K_2$.  Associated to the straight arrow is a map

\[ F: \F^{j_1} \widetilde{H}_n^i (K_1) \otimes \F^{j_2} \widetilde{H}_n^i (K_2) \longrightarrow \F^{j_1 + j_2 + n -1} \widetilde{H}_n^i (K) {\rm ,} \]

\noindent and associated to the bendy arrow is a map

\[G:  \F^j \widetilde{H}_n^i (K) \longrightarrow \bigcup_{\stackrel{j_1,j_2}{j_1+j_2 = j + n -1}} \F^{j_1} \widetilde{H}_n^i (K_1) \otimes \F^{j_2} \widetilde{H}_n^i (K_2) {\rm .}\]

For $p = 0,1,\ldots , n-1$ we write $[g_p], [g^1_p], [g^2_p]$ for Gornik's basis elements of $\widetilde{H}_n ^0(K), \widetilde{H}_n ^0(K_1), \widetilde{H}_n ^0(K_2)$ respectively.  We have

\[ F([g^1_{p_1}] \otimes [g^2_{p_2}]) = \alpha [g_{p_1}]\]

\noindent where $\alpha \not= 0$ iff $p_1 = p_2$.  And

\[ G([g_p]) = \beta ([g^1_p] \otimes [g^2_p]) \]

\noindent where $\beta \not= 0$. \qed
\end{proposition}

With this proposition in hand we are almost ready to prove Proposition \ref{kwong} and hence Theorem \ref{mainthm1}.  We just need one more easy lemma.

\begin{lemma}
If $g \in \widetilde{C}_n^0(D)$ is one of Gornik's basis elements of $\widetilde{H}_n^0(K)$ then 

\[{\rm qgr}([g]) = s_n^{\rm max}(K) {\rm .} \]
\end{lemma}

\begin{proof}
This follows from the observation that the quantum grading of exactly one of the $[h_p]$ must be $s_n^{\rm max}(K)$, and $g$ is a linear combination of the $h_p$, with all coefficients non-zero.
\end{proof}

\begin{proof}[Proof of Proposition \ref{kwong}]
In Figure \ref{connectsum}, let $K=K_1$ and let $K_2 = U$, the unknot.  Choose $p \in \{ 0,1, \ldots ,n-1 \}$ so that $[h^1_p]$ is non-zero in $\widetilde{H}_n^{0, s_n^{\rm min}}(K_1)$.  Now $h^1_p$ is expressible as a linear combination of Gornik's generators $g^1_0, g^1_1, \ldots, g^1_{n-1}$.  Assume without loss of generality that the coefficient of $g^1_0$ in this linear combination is non-zero.  Then we have

\begin{eqnarray*}
s_n^{\rm max}(K) &=& {\rm qgr}([g_0]) \\
&=& F([h^1_p] \otimes [g^2_0]) \\
&\leq& {\rm qgr}([h^1_p]) + {\rm qgr}([g^2_0]) + n -1 \\
&=& s_n^{\rm min}(K) + n - 1 + n - 1 \\
&=& s_n^{\rm min}(K) + 2n -2 {\rm .}
\end{eqnarray*}
\end{proof}

Now Theorem \ref{mainthm1} follows easily.

\begin{proof}[Proof of Theorem \ref{mainthm1}]
Propositions \ref{sonal} and \ref{kwong} combine to imply Theorem \ref{mainthm1}
\end{proof}

We can use the same technique used in the proof of Proposition \ref{kwong} to give a proof of Theorem \ref{mainthm2}.

\begin{proof}[Proof of Theorem \ref{mainthm2}]
To check we have a homomorphism, it is enough to show that $s_n$ respects the group operations.  In other words if $K = K_1 \hash K_2$ we wish to see

\[ s_n(K) = s_n(K_1) + s_n(K_2) {\rm .} \]

\noindent Again we refer to Figure \ref{connectsum} and choose $p \in \{ 0,1, \ldots ,n-1 \}$ so that $[h^1_p]$ is non-zero in $\widetilde{H}_n^{0, s_n^{\rm min}}(K_1)$ and assume without loss of generality that the coefficient of $g^1_0$ in the expansion of $h^1_p$ is non-zero.

We observe

\begin{eqnarray*}
s_n(K_1) + s_n(K_2) &=& s_n^{\rm min}(K_1) + s_n^{\rm max}(K_2) \\
&=& {\rm qgr}([h^1_p] \otimes [g^2_0]) \\
&\geq& {\rm qgr}(F([h^1_p] \otimes [g^2_0])) - n + 1 \\
&=& {\rm qgr}([g_0]) - n + 1 \\
&=& s_n^{\rm max}(K) - n + 1 \\
&=& s_n(K) {\rm ,}
\end{eqnarray*}

\noindent and

\begin{eqnarray*}
s_n(K_1) + s_n(K_2) &=& s_n^{\rm max}(K_1) + s_n^{\rm max}(K_2) - 2n +2 \\
&=& {\rm qgr}([g^1_0] \otimes [g^2_0]) - 2n + 2 \\
&=& {\rm qgr}(G([g_0])) - 2n + 2 \\
&\leq& {\rm qgr}([g_0]) + n -1 - 2n + 2 \\
&=& s_n^{\rm max}(K) - n + 1 \\
&=& s_n(K) {\rm .}
\end{eqnarray*}
\end{proof}

Finally we indicate the proof of Corollary \ref{lobbtomomi}.

\begin{proof}[Proof of Corollary \ref{lobbtomomi}]
The main tool is due to Kawamura \cite{Kaw} in which she gives an explicit estimate of $s(K)$ and $\tau(K)$ depending on a diagram $D$ of $K$.  In deriving this estimate she only makes use of the formal properties of $s$ and $\tau$ analogous to Corollary \ref{lobbwu} and Theorem \ref{mainthm2}, hence her arguments also apply to $s_n$.

In \cite{L2}, the author independently derives this estimate for $s(K)$, using an algebraic argument rather than the formal properties of $s$.  Proposition $1.5$ of \cite{L2} shows that the estimates are tight given an alternating diagram $D$ of $K$, but the proof of this Proposition does not use the definition of $s$ and hence also shows that the bounds on $s_n(K)$ are tight for alternating knots.

Therefore since we know appropriately rescaled versions of this Corollary hold for $s$ and for $\tau$, it also holds for $s_n$.
\end{proof}

\end{document}